\documentclass[12pt,a4paper,oneside]{amsart}
\usepackage{amsfonts, amsmath, amssymb, amsthm, mathtools}
\usepackage[colorlinks=true,citecolor=blue]{hyperref}
\usepackage[utopia]{mathdesign}
\usepackage[mathscr]{euscript}
\usepackage[utf8]{inputenc}

\usepackage{graphicx}
\usepackage{parskip}
\usepackage{enumerate}
\usepackage{verbatim}
\usepackage{tikz}
\usepackage{units}

\begingroup
    \makeatletter
    \@for\theoremstyle:=definition,remark,plain\do{%
        \expandafter\g@addto@macro\csname th@\theoremstyle\endcsname{%
            \addtolength\thm@preskip\parskip
            }%
        }
\endgroup

\newtheorem{theorem}{Theorem}[section]
\newtheorem*{theorem*}{Theorem}
\newtheorem{lemma}[theorem]{Lemma}
\newtheorem{prop}[theorem]{Proposition}

\theoremstyle{definition}

\newtheorem{remark}[theorem]{Remark}
\newtheorem*{remark*}{Remark}

\newcommand{\mb}[1]{\mathbb{#1}}

\newcommand{\st}{\mbox{\em st}}
\newcommand{\sta}{\mbox{st}}

\begin{document} 

\title{A Geometric Hall-type Theorem}  
% \author{Andreas Holmsen$^\dagger$ \and Leonardo Martinez-Sandoval$^\star$ \and Luis Montejano$^\diamond$}

\author{Andreas F. Holmsen}
\address{Andreas F. Holmsen, %\hfill \hfill \linebreak 
Department of Mathematical Sciences, % \hfill \hfill \linebreak
KAIST, 
Daejeon, South Korea.  \hfill \hfill }
\email{andreash@kaist.edu}

\author{Leonardo Martinez-Sandoval}
\address{Leonardo Martinez-Sandoval, % \hfill \hfill \linebreak
Instituto de Matemáticas, % \hfill \hfill \linebreak
 National University of Mexico at Quer{\'e}taro, % \hfill \hfill \linebreak 
Juriquilla,  Quer{\'e}taro 76230, Mexico,  \hfill \hfill \linebreak
 and \linebreak
 Univesit{\'e} de Montpellier, % \hfill \hfill \linebreak 
Place Eug{\'e}ne Bataillon, % \hfill \hfill \linebreak
34095 Montpellier Cedex, France.  \hfill \hfill}
\email{leomtz@im.unam.mx}

\author{Luis Montejano}
\address{Luis Montejano, % \hfill \hfill \linebreak
Instituto de Matemáticas, % \hfill \hfill \linebreak
 National University of Mexico at Quer{\'e}taro, %  \hfill \hfill \linebreak
Juriquilla , Quer{\'e}taro 76230, Mexico. \hfill \hfill}
\email{luis@matem.unam.mx}



\thanks{The first author would like to thank the Instituto de Matem{\'a}ticas, UNAM at Querétaro for their hospitality and support during his visit. The second and third authors wish to acknowledge support from CONACyT under Project 166306,  support from PAPIIT--UNAM under Project IN112614 and support from ECOS Nord project M13M01.
}

\maketitle 

\begin{abstract} We introduce a geometric generalization of  Hall's marriage theorem. For any family $F = \{X_1, \dots, X_m\}$ of finite sets in $\mb{R}^d$, we give conditions under which it is possible to choose a point  $x_i\in X_i$ for every $1\leq i \leq m$  in such a way that  the points $\{x_1,...,x_m\}\subset \mb{R}^d$ are in general position.  We give two proofs, one elementary proof requiring slightly stronger conditions, and one proof using topological techniques in the spirit of Aharoni and Haxell's celebrated generalization of Hall's theorem. 
\end{abstract}

\section{Introduction}

\subsection{Background} Let $F = \{S_1, \dots, S_m\}$ be a family of finite subsets of a common ground set $E$. A {\em system of distinct representatives} is an $m$-element subset $\{x_1, \dots, x_m\}\subset E$ such that $x_i \in S_i$ for all $1\leq i \leq m$. A classical result in combinatorics is {\em Hall's marriage theorem} \cite{hall} which states that a family $F = \{S_1, \dots, S_m\}$ has a system of distinct representatives if and only if $\left| \bigcup_{i\in I} S_i \right| \geq | I |$ for every non-empty subset $I\subset \{1, \dots, m\}$.

In 2000, Aharoni and Haxell \cite{aharoni1} presented a remarkable generalization of Hall's theorem. Let $F = \{H_1, \dots, H_m\}$ be a family of hypergraphs on a common vertex set $V$. A {\em system of disjoint representatives} is an $m$-element set $\{E_1, \dots, E_m\}$ of pairwise (vertex) disjoint edges such that $E_i\in H_i$ for all $1\leq i \leq m$. The Aharoni and Haxell result gives a sufficient condition for a family of hypergraphs to have a system of disjoint representatives, and their result reduces to the Hall's theorem in the case when the $H_i$ are 1-uniform hypergraphs. Their result was used to prove Ryser's conjecture for $3$-uniform hypergraphs \cite{aharoni}, but perhaps more importantly, their proof introduced topological techniques into this classical branch of combinatorics. The connections with topological combinatorics were further investigated and generalized in \cite{aharoni2}, \cite{ah-be}, \cite{a-b-m}, \cite{chudn}, \cite{haxell}, \cite{kahle}, \cite{meshulam}, \cite{meshulam2}.

\subsection{Our result} The purpose of this paper is to introduce a discrete geometric generalization of Hall's marriage theorem. We say that a subset $X\subset \mb{R}^d$ is in {\em general position} if every subset of size at most $d+1$ is affinely independent. Let $F = \{X_1, \dots, X_m\}$ be a family of finite sets in $\mb{R}^d$. A {\em system of general position representatives}  is a subset $\{x_1, \dots, x_m\}$ in general position such that $x_i \in X_i$ for all $1\leq i \leq m$. For a finite set $X\subset \mb{R}^d$ let $\varphi(X)$ denote the maximal size of a subset of $X$ in general position. We have the following.

\begin{theorem} \label{sgpr}
For every integer $d\geq 1$ there exists  a function $f_d : \mb{N}\to \mb{N}$ such that the following holds. Let $F = \{X_1, \dots, X_m\}$ be a family of finite sets in $\mb{R}^d$. If \[\varphi \left(\bigcup_{i\in I} X_i\right ) \geq f_d(|I|)\] for every non-empty subset $I\subset \{1, \dots, m\}$, then $F$ has a  system of general position representatives.
\end{theorem}

Notice that for $d=1$,  a set is in general position if its elements are pairwise distinct. Therefore we can set $f_1(k) = k$, in which case Theorem \ref{sgpr} reduces to Hall's theorem. 

Once the existence of the functions $f_d(k)$ has been established, a natural goal is to obtain good general upper bounds on these functions. In general we are interested in asymptotic bounds, that is,  when $d$ is fixed and the number of sets in the family $F$ grows. 

Let us illustrate how the the size of $F = \{X_1, \dots, X_m\}$ plays a role. Suppose $m \leq d+1$. We claim that if $\varphi \left(\bigcup_{i\in I} X_i\right ) \geq |I|$ for every non-empty subset $I\subset \{1, \dots, m\}$, then $F$ has a  system of general position representatives (which is the same condition as in Hall's theorem). This follows from the {\em matroid intersection theorem} due to Edmonds \cite{edmonds}. To see this, let the ground set be the disjoint union $E = X_1 \dot\cup \cdots \dot\cup X_m$. (We allow for the same point to appear in several $X_i$, but we keep track of its multiplicity.) Let $M_1$ be the matroid on $E$ whose independent sets are the affinely independent subsets, and let $M_2$ be the partition matroid induced by $X_1, \dots, X_m$. Let $r_1$ and $r_2$ be the respective rank functions. Given a subset $S\subset E$, let $I\subset \{1, \dots, m\}$ be the maximal subset such that $\bigcup_{i\in I} X_i \subset S$. We then have $r_1(S)\geq r_1\left(\bigcup_{i\in I}X_i \right)$ and $r_2(E - S) \geq m - |I|$. The matroid intersection theorem implies that $F$ has a system of general position representatives if $r_1(S) \geq |I|$ for every non-empty subset $S\subset E$. This inequality holds by our hypothesis since $r_1(S) = \min \{d+1, \varphi(S)\} \geq \min \{d+1, \varphi \left(\bigcup_{i\in I}X_i \right) \} \geq |I|$.

It is also easily seen that when $m>d+1$, the condition $\varphi\left(\bigcup_{i\in I}X_i \right) \geq | I |$ is not sufficient to guarantee a system of general position representatives. Suppose $|X_1|$ $=\cdots$ $ = |X_{m-1}| =1$ and that $\bigcup_{i=1}^{m-1}X_i$ is in general position in $\mb{R}^d$. From every hyperplane spanned by a $d$-tuple from $\bigcup_{i=1}^{m-1}X_i$ choose an additional point, at random, to be included in the set $X_m$. Thus $X_m$ consists of $\binom{m-1}{d}$ points in general position. For every non-empty subset $I\subset \{1, \dots, m\}$ we have $\varphi\left(\bigcup_{i\in I}X_i \right) \geq | I |$, but $F$ has no system of general position representatives.

\subsection{Outline of paper} We will present two proofs for the existence of the functions $f_d(k)$. The first proof uses an elementary pigeon-hole argument and gives an upper bound in $O(k^{d+1})$. This is given in Section \ref{elementary}. Our second proof uses more sophisticated techniques and gives an upper bound in $O(k^d)$. This is given in Section \ref{non-elem}, while the main auxiliary result (Theorem \ref{gencom}) is proved in Section \ref{dusty}. We do not know if this bound is optimal, and it is an interesting problem to determine better bounds on $f_d(k)$. The reader familiar with matroids will notice that many of our arguments rely on properties of the underlying matroid of the point set. This leads to generalizations of our results which will be discussed further in Section \ref{remarks}. (All matroids arising in our setting are loopless, so this will be implicitly assumed throughout.)

Just as the seminal result of Aharoni and Haxell, our second proof of Theorem \ref{sgpr} relies on topological methods, and we assume the reader is familiar with some basic notions of combinatorial topology. By using a result of Kalai and Meshulam \cite[Proposition 3.1]{kalai-mesh} (also appearing implicitly in \cite{aharoni2}, \cite{aharoni1}, and \cite{meshulam}), Theorem \ref{sgpr} can be reduced to the problem of showing that a certain simplicial complex is highly connected. We remind the reader that a topological space $X$ is {\em $k$-connected } if every map $f \colon \mb{S}^i\to X$ extends to a map $\hat{f} \colon \mb{B}^{i+1} \to X$ for $i = -1, 0, 1, \dots, k$. Here $\mb{B}^{i+1}$ denotes the $(i+1)$-dimensional ball whose boundary is the $i$-dimensional sphere $\mb{S}^i$, and $(-1)$-connected means non-empty. The following observation is sufficient for our application: {\em A simplicial complex is $k$-connected if and only if its $(k+1)$-skeleton is $k$-connected.}

Before getting to the details, let us conclude with a few words about the simplicial complex arising in our second proof of Theorem \ref{sgpr}. It was made explicit in \cite{meshulam}, that the key idea in the Aharoni and Haxell result is to capture {\em pairwise disjointness} among the members in a family of sets. This can be encoded by the {\em disjointness graph} of the family, and the resulting simplicial complex is the {\em clique complex} of the disjointness graph. However, the general position property is not a pairwise condition (for $d\geq 2$), and to encode the subsets in general position requires a simplicial complex, the {\em independence complex} of the underlying matroid of the point configuration. This in turn requires a higher-dimensional version of a clique complex, which we call the {\em completion}. A crucial observation concerning the completion of a complex is Lemma \ref{q-star}, which gives a local combinatorial condition on a simplicial complex which guarantees that its completion is $k$-connected.

\section{Proof of Theorem \ref{sgpr}} \label{elementary}

For positive integers $d$ and $k$ let
\[A_d(k) \coloneqq \begin{cases} k &\mbox{if } k\leq d+1 \\ 
d\binom{k-1}{d}+1 & \mbox{if } k > d+1.\end{cases} \]
Notice that $A_d(k)$ is in $O(k^d)$.

\begin{lemma}
  \label{lemUp}
  Let $k$ be a positive integer. If $S$ and $T$ are sets in general position in $\:\mb{R}^d$ where $|S| = k-1$ and $|T|\geq A_d(k)$, then there exists a point $p$ in $T$ such that $S\cup \{p\}$ is in general position.
\end{lemma}

\begin{proof}
For $k\leq d+1$ the result is a consequence of the augmentation property of the underlying matroid of a set of points in $\mathbb{R}^d$ (the independent sets are the affinely independent sets). Suppose now that $k\geq d+2$ and that $T$ is a set of points in general position with $|T| \geq A_d(k)$. Notice that $S$ spans $\binom{k-1}{d}$ affine hyperplanes. In each of these hyperplanes there can be at most $d$ points from $T$ since $T$ is in general position. Therefore there exists a point $p$ in $T$ which does not lie in any of these hyperplanes, implying that $S\cup \{p\}$ is in general position.
\end{proof}

Now let $B_d(k)=k(A_d(k)-1)+1.$ Notice that $B_d(k)$ is in $O(k^{d+1})$.

\begin{theorem}
Let $F=\{X_1, \ldots, X_m\}$ be a family of finite sets in $\mathbb{R}^d$. If 
\[\varphi\left(\bigcup_{i\in I} X_i\right)\geq B_d(|I|)\] for every non-empty subset $I\subset \{1,\ldots,m\}$, then $F$ has a system of general position representatives.
\end{theorem}

\begin{proof} By the hypothesis, $\varphi\left(\bigcup_{i=1}^m X_i\right) \geq m(A_d(m)-1) + 1$. By the pigeon-hole principle, there are at least $A_d(m)$ points in general position belonging to one of the sets $X_1, \dots, X_m$, so we may assume $\varphi(X_m)\geq A_d(m)$. Using the hypothesis for $I=\{1,\ldots,m-1\}$, the same reasoning implies that there are $A_d(m-1)$ points in general position belonging to one of the sets $X_1$, $\ldots$, $X_{m-1}$, so we may assume $\varphi(X_{m-1}) \geq A_d(m-1)$. Proceeding downwards we may assume that $\varphi(X_i) \geq A_d(i)$ for each $i\in \{1,\ldots,m\}$. 
	
Now we use Lemma \ref{lemUp} upwards. We take a point $p_1 \in X_1$. Suppose we have selected points $p_i\in X_i$ for $i\in \{1,2,\ldots,k-1\}$ such that $\{p_1,\ldots,p_{k-1}\}$ is in general position. Then Lemma \ref{lemUp} allows us to select a point $p_k \in X_k$ such that $\{p_1,\ldots,p_k\}$ is in general position. We continue up to $k=m$ to get the desired system of general position representatives.  
\end{proof}

\section{A better upper bound by topological methods}\label{non-elem}

\subsection{The general position complex}
Let $X\subset \mb{R}^d$ be a finite (multi)set. Let us define the {\em general position complex} of  $X$, denoted by $G(X)$, to be the simplicial complex \[G(X) \coloneqq \{S \subset X \; : \; S \mbox{ is in general position in } \mb{R}^d \}.\] Note that we allow for $X$ to have repeated points. The number of vertices of $G(X)$ is the cardinality of $X$, counting multiplicities. A key observation is that the connectivity of $G(X)$ can be bounded below in terms of $d$ and $\varphi(X)$.

\begin{theorem} \label{gencom} 
For all integers $d\geq 1$ and $k\geq -1$ there exists a minimal positive integer $g_d(k)$ such that the following holds. If $X\subset \mb{R}^d$ is a finite (multi)set with $\varphi(X) \geq g_d(k)$, then $G(X)$ is $k$-connected.
\end{theorem}

A closely related simplicial complex is the {\em independence complex} of $X$, denoted by $M(X)$, defined as \[M(X) \coloneqq \{S \subset X \: : \: S  \mbox{ is affinely independent}\}.\] The simplices of $M(X)$ are the independent sets of a matroid, the underlying matroid of $X$, which has rank $r = \min \{\varphi(X) , d+1\} = \dim M(X) + 1$. Note that $M(X)$ is the $(r-1)$-skeleton of $G(X)$.

\begin{remark} \label{2nd remark} We postpone the proof of Theorem \ref{gencom}, but here we note the following special cases. 
\begin{itemize}
\item For $d = 1$, a multiset $X$ consists of $n = \varphi(X)$ distinct points with mutliplicities $m_1, \dots, m_n$. The corresponding general position complex is the join of $n$ discrete sets of points. That is, $G(X) = V_1 \ast \cdots \ast V_n$, where $|V_i| = m_i$. If $|V_i|=1$ for any $i$, then $G(X)$ is contractible. If $|V_i| > 1$ for all $i$ it is known that $G(X)$ is homotopic to a wedge of $(n-1)$-dimensional spheres which is $(n-2)$-connected. Therefore $g_1(k) = k+2$. 
\item If $k \leq d-1$, then $g_d(k) = k+2$. In this case $G(X) = M(X)$, and the claim follows from the well-known fact that the independence complex of a rank $r$ matroid is $(r-2)$-connected (see e.g. \cite{bjorn, BjKoLo}).
\end{itemize}
\end{remark}

\subsection{Colorful simplices}
Let $K$ be a simplicial complex on the vertex set $V$, and let $V = V_1 \cup \cdots \cup V_m$ be a partition. A simplex $S\in K$ is called {\em colorful} if $|S\cap V_i| = 1$ for all $1\leq i \leq m$. For a non-empty subset $I \subset \{1, \dots, m\}$ let $K(I)$ denote the {\em induced subcomplex} $K\left[\bigcup_{i\in I}V_i\right]$. 

The following sufficient condition for the existence of a colorful simplex in $K$ was given in \cite[Proposition 3.1]{kalai-mesh} (where it is stated in terms of rational homology rather than connectedness), and in a more general form in \cite[Theorem 4.5]{aharoni2}.

\begin{prop}[Kalai and Meshulam] \label{colorful}
  Let $K$ be a simplicial complex on the vertex set $V$ with partition $V = V_1 \cup \cdots \cup V_m$. If the induced subcomplex $K(I)$ is $(|I|-2)$-connected for every non-empty subset $I\subset\{1, \dots, m\}$, then $K$ contains a colorful simplex.
\end{prop}

\begin{proof}[Second proof of Theorem \ref{sgpr}]
Let $F = \{X_1, \dots, X_m\}$ be a family of finite sets in $\mb{R}^d$ and let $X = X_1 \dot\cup \cdots \dot\cup X_m$ (that is, counting multiplicities). The members of $F$ induce a partition of the vertex set of $G(X)$, and $F$ has a system of general position representatives if and only if the general position complex $G(X)$ contains a colorful simplex. If $\varphi\left( \bigcup_{i\in I} X_i\right) \geq g_d(|I|-2)$ for every non-empty subset $I\subset \{1, \dots, m\}$, then, by Theorem \ref{gencom}, $G(X)$ satisfies the conditions of Proposition \ref{colorful}. Therefore $f_d(k)$ can be bounded above by  $g_d(k-2)$. 
\end{proof}

\begin{remark}
  In the next section we give an upper bound on $g_d(k)$ which is in $O(k^d)$.
\end{remark}

\section{Proof of Theorem \ref{gencom}}\label{dusty}

\subsection{The completion of a simplicial complex}
Let $k$ be a positive integer and $S$ a finite set with $|S| \geq k$. The collection of all subsets of $S$ of size at most $k$ is denoted by \[[S]_{k} \coloneqq \{T\subset S \: : \: |T|\leq k\}.\] Let us also define $[S]_0 \coloneqq \emptyset$.

Let $K$ be a simplicial complex of dimension $d$ on the vertex set $V$. For the proof of Theorem \ref{gencom} we need the following simplicial complexes associated with $K$.

\bigskip

For a vertex $v\in V$, let $\sta_K(v)$ denote the {\em star} of $v$, which is defined as 
\[\sta_K(v) \hspace{1ex} \coloneqq \hspace{1ex} \{ S\subset V \: : \: S\cup \{v\} \in K \}.\] Notice that $\sta_K(v)$ is always non-empty since $v\in \sta_K(v)$. Also, if $L$ is a subcomplex of $K$ and $v$ is a vertex of $L$, then $\sta_L(v) \subset \sta_K(v)$.

\bigskip

For a vertex $v\in V$, let $\Gamma_K(v)$ denote the {\em neighborhood complex} of $v$, which is defined as
 \[\Gamma_K(v) \hspace{1ex} \coloneqq \hspace{1ex} \mbox{st}_K(v) \hspace{1ex} \cup \hspace{1ex} \{S \subset V - \{v\} \: :\: S \in K, |S| = d+1, [S]_d \subset \mbox{st}_K(v)\}.\] 
We warn the reader about the subtle dependence on $d = \dim K$. For instance, if $K$ is $0$-dimensional, i.e. a set of isolated vertices, then $\Gamma_K(v) = K$. However, if $K$ has positive dimension and $v$ is an isolated vertex of $K$, then $\Gamma_K(v) = \{v\}$. This shows that if $L$ is a subcomplex of $K$, then it is not generally true that $\Gamma_L(v)$ is a subcomplex of $\Gamma_K(v)$.

\bigskip

For $j\geq d$, let $\Delta_j(K)$ denote the {\em $j$-completion} of $K$, which is defined as \[\Delta_j(K) \hspace{1ex} \coloneqq \hspace{1ex} K \hspace{1ex} \cup \hspace{1ex} \{ S \subset V \: : \: |S|\geq j+2 , [S]_{j+1}\subset K\}.\] Let us also define $\Delta_j(\emptyset) \coloneqq \emptyset$. Notice that if $K$ is $0$-dimensional then $\Delta_0(K)$ is the $(|V|-1)$-dimensional simplex. Also, if $j >\dim K$, then $\Delta_j(K) = K$. Consequently, if $L$ is a subcomplex of $K$, then $\Delta_d(L) \subset \Delta_d(K)$. 

\begin{prop} \label{iden1}
Let $K$ be a simplicial complex of dimension $d$ and let $\{K_i\}_{i\in I}$ be a finite family of subcomplexes of  $K$. Then \[\Delta _d\left(\bigcap_{i\in I} K_i \right) \; = \; \bigcap_{i\in I} \Delta_d(K_i).\]
\end{prop}

\begin{proof} Since $\bigcap_{i\in I} K_i \subset K_i$, we have $\Delta_d\left(\bigcap_{i\in I}K_i\right)\subset  \Delta_d(K_i)$. Therefore $\Delta_d\left(\bigcap_{i\in I}K_i\right) \subset \bigcap_{i\in I}\Delta_d(K_i)$. 

For the other direction, suppose $S\in \bigcap_{i\in I}\Delta_d(K_i)$.  If $|S| \leq d+1$, then $S \in \bigcap_{i\in I}K_i\subset  \Delta_d\left(\bigcap_{i\in I}K_i\right)$. If $|S| \geq d+2$, then  $[S]_{d+1} \subset K_i$ for every $i\in I$. That is, $[S]_{d+1} \subset \bigcap_{i\in I} K_i$, and therefore $S \in  \Delta_d\left(\bigcap_{i\in I}K_i\right)$.
\end{proof}

\begin{prop}\label{iden2}
Let $K$ be a simplicial complex of dimension $d$ on the vertex set $V$ and let $v\in V$. Then \[\st_{\Delta_d(K)}(v) \; = \; \Delta_d\left(\Gamma_K(v)\right).\]
\end{prop}

\begin{proof} We first show that $\sta_{\Delta_d(K)}(v) \subset \Delta_d(\Gamma_K(v))$.  Suppose $S \in \sta_{\Delta_d(K)}(v)$, which, by definition, means that \[S \cup \{v\} \; \in \; \Delta_d(K) \; = \; K \; \cup \; \{T \subset V \: : \: |S|\geq d+2 , [S]_{d+1}\subset K\}.\] If $|S\cup \{v\}|\leq d+1$, then $S \cup \{v\}\in K$. This implies that $S\in \sta_K(v)$, and since  $\sta_K(v) \subset \Gamma_K(v) \subset \Delta_d(\Gamma_K(v))$ we have $S\in \Delta_d(\Gamma_K(v))$.

If $|S \cup \{v\}| \geq d+2$, then $[S \cup \{v\}]_{d+1} \subset K$. This implies that for every $T\in [S - \{v\}]_{d}$ we have $T  \in \sta_K(v)$, and
%Thus, for every $T \in [S]_{d+1}$, if $v\in T$, then $T\in \sta_K(v) \subset \Gamma_K(v)$, and if $v \notin T$, then $T\in \Gamma_K(v)$. 
consequently $[S]_{d+1} \subset \Gamma_K(v)$. Therefore $S\in \Delta_d(\Gamma_K(v))$.

\bigskip

It remains to show that $\Delta_d(\Gamma_K(v)) \subset \sta_{\Delta_d(K)}(v)$. Suppose $S \in \Delta_d(\Gamma_K(v))$. 

If $|S|\leq d+1$, then $S \in\Gamma_K(v)$. Furthermore, if $|S - \{v\}| \leq d$ it follows that $S \in \sta_K(v)$. Since $K \subset \Delta_d(K)$ we have $S \in \sta_{\Delta_d(K)}(v)$. On the other hand, if $v\notin S$ and $|S| = d+1$, then, by definition, we have $S\in K$ and $[S]_d \subset \sta_K(v)$. This implies that  $[S \cup \{v\}]_{d+1} \subset K$, and it follows that $S \cup \{v\} \in \Delta_d(K)$, which shows that $S \in \sta_{\Delta_d(K)}(v)$. 

If $|S|\geq d+2$, then $[S]_{d+1} \subset \Gamma_K(v)$. This implies that for every $T\in [S - \{v\}]_d$ we have $T \in \sta_K(v)$, and for every $T \in [S - \{v\}]_{d+1}$ we have $T \in K$. It follows that $[S \cup \{v\}]_{d+1} \subset K$, and therefore $S \cup \{v\} \in \Delta_d(K)$, which shows that $S \in \sta_{\Delta_d(K)}(v)$.
\end{proof}

\subsection{The Nerve theorem}\label{nerves}
Let $F$ be a finite family of sets. The {\em nerve} of $F$, denoted by $N(F)$, is the simplicial complex on the vertex set $F$ whose simplices are the intersecting subfamilies of $F$, that is \[N(F)  \coloneqq \{ G\subset F \: : \: \bigcap_{S\in G}S \neq \emptyset\}.\] We will use the following version of the Nerve theorem which is a consequence of \cite[Theorem 6]{bjorner}.

\begin{theorem}[Bj{\"o}rner]\label{nerve}
  Let $K$ be a simplicial complex and $F = \{K_i\}_{i \in I}$ a finite family of subcomplexes such that $K = \bigcup_{i\in I}K_i$. Suppose every non-empty intersection $\bigcap_{t\in T} K_t$ is $(k + 1 - |T|)$-connected, $T \subset I$. Then $K$ is $k$-connected if and only if $N(F)$ is $k$-connected. 
\end{theorem}

\subsection{The $q$-star property}
Let $K$ be a simplicial complex of dimension $d$ on the vertex set $V$. For any integer $q\geq 1$, we say that $K$ is {\em $q$-star} if $|V|>q$, and for every subset $Y \subset V$ of size $q$ there exists a vertex $v\in V - Y$ such that  $S \cup \{v\} \in K$ for every simplex $S \in K[Y]$ with $|S|\leq d$.  The following is an extension of a result on clique complexes appearing in \cite[Theorem 3.1]{kahle}, and in a more general form in \cite[Theorem 1.5]{meshulam}.

\begin{lemma} \label{q-star}
Let $K$ be a simplicial complex of dimension $d$ and let $k$ be a non-negative integer. If $K$ is $(2k+2)$-star, then its $d$-completion  $\Delta_d(K)$ is $k$-connected.
\end{lemma}

\begin{proof} For $d=0$ the statement holds because $\Delta_0(K)$ is a simplex which is contractible. We may therefore assume $d\geq 1$.

If a complex $K$ of dimension $d\geq 1$ is $2$-star, then $K$ is connected which implies that $\Delta_d(K)$ is also connected. So the statement is clearly true for $k=0$, and we proceed by induction on $k$. 

Suppose $K$ is $(2k+2)$-star for $k>0$ and let $V$ be the vertex set of $K$. For each vertex $v\in V$, let $K_v = \sta_{\Delta_d(K)}(v)$. Define the family of subcomplexes $F = \{K_v\}_{v\in V}$. Clearly we have \[\Delta_d(K) = \bigcup_{v\in V} K_{v}.\] For a non-empty subset $W\subset V$, let $K_W = \bigcap_{v\in W}K_v$. Theorem \ref{nerve} implies that $\Delta_d(K)$ is $k$-connected, if we can show the following.
\begin{enumerate}[(i)]
\item $K_W$ is $(k + 1 - t)$-connected for every non-empty subset $W\subset V$ with $|W| = t$.
\item The nerve $N(F)$ is $k$-connected.
\end{enumerate}

\bigskip

Part (i). For every $v\in V$, $K_v$ is a cone which is contractible, and hence $k$-connected. Now consider $W\subset V$ with $|W| = t \geq 2$, and let $L_W = \bigcap_{v\in W} \Gamma_K(v)$. By Propositions \ref{iden1} and \ref{iden2} we have \[K_W \; = \; \bigcap_{v\in W}\sta_{\Delta_d(K)}(v) \; =\; \Delta_d\left(\bigcap_{v\in W} \Gamma_K(v)\right) \; = \; \Delta_d(L_W).\] By induction, it therefore suffices to prove that $L_W$ is $(2(k + 1 - t) + 2)$-star. Also notice that for $t\geq 2$ we have $2k + 2 - t \geq 2(k + 1 - t) + 2$, so it suffices to show that $L_W$ is  $(2k+2- t)$-star.

Let $X$ be the vertex set of $L_W$. Clearly a vertex $v$ belongs to $X$ if and only if $\{v,w\} \in K$ for all $w\in W$. This implies that $|X| > 2k+2-t$, since $K$ is $(2k+2)$-star.

Next, we observe that for every $Y \subset X$ with $|Y| = 2k + 2 - t$ we can find a set $Z\subset X \cup W$ with $|Z| = 2k+2$ such that $Y\cup W \subset Z$. Since $K$ is $(2k+2)$-star, there exists $v\in V - Z$ such that \[S\cup \{v\} \in K \; \mbox{ for every } \; S\in K[Z] \; \mbox{ with } \; |S|\leq d. \label{eq:st} \tag{$\ast$}\] It follows from our previous observation that $v\in X$.

Let $S\in L_W[Y]$ with $|S| \leq \dim L_W \leq d$. We need to show that $S\cup \{v\} \in L_W$, that is, $S\cup \{v\} \in \Gamma_K(w)$ for every $w\in W$. Notice that $S \cup \{w\} \subset Z$, so \eqref{eq:st} may be applied provided $|S\cup \{w\}|\leq d$.

If $|S \cup \{w\}|\leq d$, then \eqref{eq:st} implies that $S\cup \{w\} \cup \{v\} \in K$. This just means that $S\cup \{v\} \in \sta_K(w)$, and consequently $S \cup \{v\} \in \Gamma_K(w)$. 

If $|S\cup \{w\}| = d+1$, then $|S| = d$,  $w \notin S$, and $S \cup \{w\} \in K$. For every $T\in [S \cup \{w\}]_d$, it follows from \eqref{eq:st} that $T \cup \{v\} \in K$. In particular $S \cup \{v\} \in K$ and $[S\cup \{v\}]_d\subset \sta_K(w)$, and since $|S \cup \{v\}| = d+1$, we conclude that $S\cup \{v\} \in \Gamma_K(w)$. 

This shows that $L_W$ is $(2k + 2 - t)$-star.

\bigskip

Part (ii). Clearly $K_W$ is non-empty for any subset $W\subset V$ with $|W| = 2k+2$. Therefore the $(2k+1)$-skeleton of the nerve $N(F)$ is complete, which implies that $N(F)$ is $2k$-connected.\end{proof}

\begin{proof}[Proof of Theorem \ref{gencom}] 
Let $K = M(X)$, the independence complex of $X$. Clearly the general position complex of $X$ is the $d$-completion of $K$, that is, $G(X) = \Delta_d(K)$. We want to show that $G(X)$ is $k$-connected provided $\varphi(X)$ is sufficiently large. In view of Remark \ref{2nd remark}, we may assume that $k\geq d$. We will show that if $\varphi(X) > d\binom{2k+2}{d}$, then $K$ is $(2k+2)$-star. This implies that $G(X)$ is $k$-connected by Lemma \ref{q-star}. Let $S\subset X$ with $|S| = 2k+2$. Let $H$ denote the set of hyperplanes spanned by affinely independent $d$-tuples in $S$. Therefore, if $\varphi(X) >  d\binom{2k+2}{d} \geq d|H|$, then there exists a point $x\in X$ which is not contained in the affine hull of any subset $T\subset S$ with $|T|\leq d$, and consequently $K$ is $(2k+2)$-star. This gives an upper bound on $g_d(k)$ which is in $O(k^d)$.\end{proof}

\section{Concluding remarks} \label{remarks}
The natural problem that arises is to try to determine better (or exact) bounds for the functions $g_d(k)$. We have shown that $g_d(k) = k+2$ for $k \leq d-1$ and $g_d(k) \leq d\binom{2k+2}{d}+1$, otherwise, but we hardly believe this to be optimal. In fact, the exact same proof  (and bound) works in a more general setting, which we now describe.  

Let $M$ be a matroid of rank $r$ on the ground set $E$. We say that a subset $S\subset E$ is {\em uniform} if $S$ is independent or $|S| >r$ and every member of $[S]_r$ is independent. The set of all uniform subsets of a matroid $M$ form a simplicial complex, which call the {\em uniformity complex} of $M$. Obviously, the uniformity complex of a matroid is the $(r-1)$-completion of its independence complex. If we let $\mu(M)$ denote the maximum size of a uniform subset of $M$, then we have the following generalization of Theorem \ref{gencom}.

\begin{theorem}\label{uniformity}
For all integers $r\geq 2$ and $k\geq -1$ there exists a minimal positive integer $h_r(k)$ such that the following holds. If $M$ is matroid of rank $r$ and $\mu(M) \geq h_r(k)$, then the uniformity complex of $M$ is $k$-connected.
\end{theorem}

\begin{proof}
  The same argument (as in the proof of Theorem \ref{gencom}) shows that if $\mu(M) > (r-1) \binom{2k+2}{r-1}$, then the independence complex of $M$ is $(2k+2)$-star. The theorem then follows from Lemma \ref{q-star}.
\end{proof}

We find it likely that there should be a sharp distinction in the asymptotic behavior between the function $h_r(k)$ and the corresponding function $g_{r-1}(k)$. More generally, we find it reasonable to expect the orientability of the matroid $M$ to have a strong quantitative effect on the connectivity of the uniformity complex, but we lack any evidence to support this. In fact, the only exact value we know (apart from what is covered by Remark \ref{2nd remark}) is $g_2(2) = h_3(2) = 7$.

In conclusion we mention that, in view of Theorem \ref{uniformity}, it is straightforward to apply Proposition \ref{colorful} to obtain an analogue of Theorem \ref{sgpr} for {\em uniform systems of  representatives}. Further generalizations can also be obtained by using the more general version of Lemma \ref{colorful} appearing in \cite[Theorem 4.5]{aharoni2}. We leave the details to the reader.

\section{Acknowledgments}
We are grateful to the anonymous referee for pointing out a mistake in our original proof and for making several other valuable comments and suggestions.

\end{document}